\documentclass[12pt]{ijuc}

\usepackage{hyperref}
\usepackage[english]{babel}
\usepackage[utf8]{inputenc}
\usepackage{index}
\usepackage{graphicx}
\usepackage{amsfonts,amsmath,amssymb,amsthm,enumerate,latexsym}

\let\phi\varphi

\newcommand\tupl[1]{\overline{#1}}

\let\subset\subseteq
\def\compNP{{\textsf{NP}}}
\def\compEXPTIME{{\textsf{EXPTIME}}}
\def\compP{{\textsf{P}}}

\def\algA{{\mathbf{A}}}
\def\algB{{\mathbf{B}}}

\newcommand\vect[1]{\overline{#1}}

\def\en{{\mathbb N}}

\let\epsilon\varepsilon

\theoremstyle{plain}
\newtheorem{theorem}{Theorem}
\newtheorem{lemma}[theorem]{Lemma}
\newtheorem{corollary}[theorem]{Corollary}
\newtheorem{observation}[theorem]{Observation}
\newtheorem{proposition}[theorem]{Proposition}

\theoremstyle{definition}
\newtheorem{definition}[theorem]{Definition}

\newtheorem{remark}[theorem]{Remark}

\author{Alexandr Kazda\email{alex.kazda@gmail.com}\footnote{ORCID
0000-0002-7338-037X}}
\institute{Department of Algebra, MFF, Charles University, Czechia}
\title{Deciding the existence of quasi weak near unanimity terms in
finite algebras}

\def\received{Received 14 February 2020; In final form XXXX}
\begin{document}
\maketitle
\begin{abstract}
We show that for a fixed positive integer $k$ one can
efficiently decide if a finite algebra $\algA$ admits a $k$-ary weak near unanimity
operation by looking at the local behavior of the terms of
$\algA$. We also observe that the problem of deciding if a given
finite algebra has a quasi Taylor operation is solvable in
polynomial time by looking, essentially, for local quasi Siggers operations.
\end{abstract}

\section{Introduction}
Maltsev conditions, the functional equations that have a solution in a
given algebra, serve as useful lens through which to view the behavior of
algebras and varieties. Classically, various properties
of congruence lattices of varieties are equivalent to Maltsev
conditions~\cite{conm}. More recently, height 1 Maltsev conditions
turned out to describe the
complexity of the non-uniform constraint satisfaction
problem~\cite{wonderland}.

Given a finite algebra and a fixed Maltsev condition, we would like to decide
if the algebra has a term satisfying the Maltsev condition. Since Maltsev
conditions tell us much about the algebra in question, a practical test for
Maltsev conditions is useful when examining concrete algebras. For example, the
Universal Algebra Calculator~\cite{UACalc} program can, among other things,
test for various Maltsev conditions including the presence of a $k$-ary weak
near unanimity term. (The Calculator does not at the moment of this
writing implement the algorithms presented in this paper, however.)

At the moment, the only known widely applicable method to test algebras for
Maltsev conditions is to check whether the Maltsev condition holds locally and
then, hopefully, put together the local pieces into a term or terms that
satisfy the Maltsev condition globally. This ``local to global'' method works
for a wide spectrum of Maltsev conditions in idempotent algebras
\cite{freese-valeriote-maltsev-conditions, horowitz-ijac, kazda-valeriote}
(however, see~\cite{minority-report} for a case where local terms are not
enough to construct a global one).

Characterizing the Maltsev conditions for which the local to global method works is
an open problem. In this paper we show that the local to global method works when
deciding whether the input finite
algebra $\algA$ admits a $k$-ary quasi weak near
unanimity operation ($k$-qWNU) for a fixed $k$. In this situation, we can check for local
$k$-qWNU operations in time polynomial in the size of $\algA$ and thus
we obtain an efficient algorithm for deciding if a given algebra has a
$k$-qWNU operation. We also explain how the related problem of deciding if an input
algebra has a quasi Taylor operation can be solved by a local to
global algorithm.

The ``quasi'' Maltsev conditions do not force the operations in
question to be idempotent and our result does not require that the input
algebra is idempotent. This is in contrast to the numerous examples of
idempotent Maltsev conditions that are $\compEXPTIME$-complete to
decide in general finite
algebras~\cite{freese-valeriote-maltsev-conditions, horowitz-ijac}.
We speculate that the ``local to global'' construction will remain
useful for general algebras as long as we avoid the need to construct
an idempotent operation out of non-idempotent basic operations. In
particular we believe that the local to global method should work 
for many height 1 Maltsev conditions other than having a qWNU or a
quasi Taylor operation.

Finally, let us remark that our complexity results do not translate into the
situation when an algebra $\algA$ is given not by its basic operations, but
rather by a set of invariant relations (that is, operations of $\algA$ are the
polymorphisms of a given relational structure). In such a situation it is
expensive to compute subpowers generated by tuples and hence the
straightforward polynomial time algorithm from Theorem~\ref{thmWNU} becomes
inefficient. In fact, Chen and Larose proved~\cite[Theorem
6.2]{chen-larose-meta} that deciding if a given relational structure satisfies
any nontrivial, consistent, strong linear Maltsev condition of height 1 is
\compNP-complete; this result covers both quasi Siggers and $k$-qWNU operations
for a fixed $k$. On the other hand, \cite[Theorem
6.2]{chen-larose-meta} does not extend to having a $k$-ary 
weak near unanimity operation because idempotence is not a height 1 condition. As far
as we know, the complexity of deciding whether a given relational structure has
a $k$-ary weak near unanimity operation is an open problem.

\section{Preliminaries}

An $n$-ary \emph{operation} on a set $A$ is a mapping $f\colon A^n\to A$.
An \emph{algebra} $\algA$ consists of non-empty set $A$, called the
\emph{universe} or \emph{domain} of $\algA$, and a set of \emph{basic
operations} $f_i$ where $i$ ranges over some index set $I$. In this
paper we will only consider algebras with finite universes. If $\tau$ is a
unary operation on $A$ and $t$ an $n$-ary operation on $A$ we will denote by 
$\tau\circ t$ the composition $\tau\circ t(x_1,\dots,x_n)=\tau(t(x_1,\dots,
x_n))$.

A \emph{term} is a valid formal expression that describes an operation as a composition of
basic operation and variable symbols. Examples of terms are
``$f(x,g(y,y,x),h(z))$'' or ``$x$.'' For a formal definition of terms,
see~\cite[Definition 10.1]{burris-sankapanavar}. If the symbols in a term $t$ are the basic
operations of an algebra $\algA$, we can evaluate $t$ in $\algA$ and obtain an
operation. Operations that we can construct in this way are called the \emph{term
operations} of $\algA$. We will call the set of all term operations of $\algA$ 
the \emph{clone} of $\algA$. If the clone of $\algA$ contains a particular kind of
operation, we will often simply say that $\algA$ has this kind of operation.

An \emph{identity} or an \emph{equation} is the statement $t\approx u$
where $t$ and $u$ are terms. The identity $t\approx u$ holds in an
algebra $\algA$ if $t$ and $u$ evaluate to the same operation in
$\algA$ (that is, the identity holds for all values of the input
variables).

A strong \emph{Maltsev condition} is a finite system of identities. An algebra
$\algA$ satisfies a strong Maltsev condition $\Sigma$ if we can choose for each
operation symbol in $\Sigma$ some operation in $\algA$ so that all the
identities of $\Sigma$ hold. A strong Maltsev condition is of \emph{height 1} if both the left
and the right hand side of all identities in the Maltsev condition 
contain exactly one term symbol (i.e. both sides are of the form ``term(variables)'').

An operation $t$ on a set $A$ is \emph{idempotent} if the identity $t(x,\dots,x)\approx x$
holds. An algebra $\algA$ is \emph{idempotent} if each of its basic operations is idempotent.

A $k$-ary operation $w$ is called a \emph{weak near unanimity} (WNU)
operation if it is idempotent and satisfies the chain of identities
\[
  w(y,x,\dots,x,x)\approx w(x,y,\dots,x,x)\approx\dots\approx
  w(x,x,\dots,x,y).
\]
If we drop the idempotence requirement, we will call $w$ a \emph{quasi weak
near unanimity} (qWNU).

An operation is a \emph{Taylor operation} if it is idempotent and
satisfies a system of identities of the form 
\begin{align*}
  t(x,?,?,\dots,?)&\approx t(y,?,?,\dots,?)\\
  t(?,x,?,\dots,?)&\approx t(?,y,?,\dots,?)\\
		  &\vdots\\
  t(?,?,?,\dots,x)&\approx t(?,?,?,\dots,y),
\end{align*}
where $x,y$ are variables and the question marks stand for some choice of $x$'s and $y$'s.
As before, if we do not require $t$ to be idempotent, we will call $t$ a \emph{quasi
Taylor} operation.

It is a classic result by W. Taylor~\cite[Corollary 5.2 and 5.3]{taylor} that an idempotent variety
$\mathcal V$ interprets into the trivial algebra on two elements if and only if
$\mathcal V$ does not have a Taylor term operation. Obviously, a $k$-WNU
operation is a special case of a Taylor operation. Much less obviously, it turns out that
having a (quasi) Taylor operation implies having a (quasi) weak near unanimity
operation of some arity as well as a specific arity 4 (quasi) Taylor term. For idempotent
algebras this was shown in~\cite{maroti-mckenzie-wnu, siggers-original}. We state this result as a
variant of~\cite[Theorem 1.4]{wonderland} here (in particular, we
replace cyclic terms with the weaker WNU terms):
\begin{theorem}\label{thmTaylorOmnibus}
  Let $\algA$ be a finite algebra. Then the following are equivalent:
  \begin{enumerate}[a)]
    \item $\algA$ has a quasi Taylor term,
    \item there exists $k\geq 2$ such that $\algA$ has a $k$-qWNU term,
    \item $\algA$ has a quasi Siggers term, i.e. an arity 4 term $s$ satisfying the
      identity
      \[
	s(r,a,r,e)\approx s(a,r,e,a).
	\]
  \end{enumerate}
  Moreover, if $\algA$ is idempotent, the ``quasi'' qualifier can be
  dropped in the first two points and $s$ is idempotent in the last
  point.
\end{theorem}

The following lemma will allow us to transfer properties between general and idempotent
algebras. This is not really a new result; rather, it is a different way to
express some ideas present in~\cite{wonderland} and~\cite[Lemma
6.4]{chen-larose-meta} in a style similar
to the construction of minimal algebras in the tame congruence theory~\cite{conm}.
\begin{lemma}\label{lem:idempotent}
  Let $\algA$ be a finite algebra with the universe $A$. Then there exists a finite \emph{idempotent}
  algebra $\algB$ with universe $B\subset A$ such that:
  \begin{enumerate}[a)]
    \item There exists a unary term operation $\alpha$ of $\algA$
      whose image is
      $B$ and such that $\alpha\circ\alpha=\alpha$,
    \item\label{itm:h1} for each strong Maltsev condition $\Sigma$ of
      height 1 we have that $\algA$ satisfies
      $\Sigma$ if and only if $\algB$ satisfies $\Sigma$, and
    \item\label{itm:composition} for each operation $t$ in the clone of term operations of $\algA$
      there exists a unary term operation $\tau$ of $\algA$ such that
      $\tau\circ t$
      restricted to $B$ lies in the clone of term operations of $\algB$.
  \end{enumerate}
\end{lemma}
\begin{proof}
  Let $\alpha$ be a unary term operation of $\algA$ such that the image of
  $\alpha$ is inclusion minimal among all images of unary operations
  in the clone of $\algA$. Denote the image set of $\alpha$ by $B$. Since $\alpha$ restricted to $B$ needs to be a permutation (or else
  $\alpha^2$ would have a smaller image than $B$), we can replace
  $\alpha$ by its suitable power so that $\alpha$ restricted to $B$ is the identity map.
  This gives us the identity $\alpha\circ\alpha=\alpha$.
  
  We construct the algebra $\algB$ on $B$ as follows: For each term
  operation $t$ of $\algA$ consider the operation $u=\alpha\circ t$
  restricted to $B$. If $u$ is idempotent, we make it into a basic
  operation of $\algB$. By definition $\algB$ is then an idempotent
  algebra. 
 
  Let us now prove~(\ref{itm:composition}). Let $t$ be an operation in
  the clone of $\algA$. Consider the mapping $\beta\colon
  B\to B$ given by $\beta(b)=\alpha(t(b,b,\dots,b))$. From the
  minimality of $B$, we get that the image of $\beta$ needs to be
  exactly $B$ and that $\beta$ is a permutation. Let $n$ be
  such that $\beta^{n+1}$ is the identity mapping. Then the operation 
  \[
    u(x_1,\dots,x_n)=\beta^n\circ\alpha\circ t(x_1,\dots,x_n)
  \]
  restricted to $B$ is a basic operation of $\algB$. Taking $\tau=
  \beta^n\circ \alpha$ gives us part~(\ref{itm:composition}).

  To prove the ``only if'' part of item~(\ref{itm:h1}), let $\Sigma$ be
  a strong Maltsev condition of height 1 and let $t_1,\dots,t_k$ be operations
  from the clone of operations of $\algA$ that satisfy $\Sigma$. We know that
  for each $t_i$ there is a $\tau_i$ such that $\tau_i\circ t_i$ restricted to
  $B$ lies in the clone of $\algB$. Moreover, from the proof of
  part~(\ref{itm:composition}) we see that
  $\tau_i$ depends only on the mapping $t_i(x,x,\dots,x)$. From this it follows
  that we can choose $\tau_1,\dots,\tau_k$ so that when $t_i$ and $t_j$ appear in one identity in $\Sigma$ we must have
  $\tau_i=\tau_j$. This accomplished, it is straightforward to verify
  that the operations $\tau_1\circ t_1,\tau_2\circ
  t_2,\dots, \tau_k\circ t_k$ satisfy $\Sigma$ and the implication is
  proved.

  Finally, let us prove the ``if'' part of item~(\ref{itm:h1}). Let $\Sigma$ be
  a strong Maltsev condition of height 1. Assume that $\algB$ has operations 
  $t_1,\dots,t_k$ that satisfy $\Sigma$. Then it is easy to verify that
  composing $t_1,\dots, t_k$ with $\alpha$ from the inside (that is,
  considering operations of the form $t_i(\alpha(x_1),\alpha(x_2),\dots)$) gives us
  operations from the clone of operations of $\algA$ that also satisfy
  $\Sigma$.
\end{proof}

A subset $B$ of $A$ is a
\emph{subuniverse} of $\algA$ if it is closed under the basic
operations of $\algA$ (equivalently, under all term operations of
$\algA$). The \emph{subuniverse generated by $X\subset A$} is the
smallest subuniverse of $\algA$ that contains the set $X$. It is easy
to see that $a$ lies in the subuniverse of $\algA$ generated by $X$ if
an only if there is a term operation of $\algA$ that outputs $a$ when applied to
the list of elements of $X$.

The $n$-th power of $\algA$ is the algebra with universe $A^n$ and
operations $f_i$ acting coordinate-wise on $A^n$. A subuniverse of a
power of $\algA$ is called a \emph{subpower of $\algA$}.

A \emph{tuple} of elements of $A$ is a member of the set $A^n$. To
emphasize that $\tupl u$ is a tuple, we will sometimes put a bar above
it. In Section~\ref{sec:main} we will use a notation such as ``$(\tupl u,
\tupl v, e)$'' to denote a tuple, whose prefix is $\tupl u$, followed by
the sequence of elements from $\tupl v$ and finally the element $e\in
A$.

An \emph{$n$-ary relation} $R$ over a set $A$ is a subset of $A^n$. A relation
is \emph{admissible} for (or compatible with) $\algA$ if $R$ is a subuniverse
of $\algA^n$; in other words if $R$ is closed under the term operations of
$\algA$ applied coordinate-wise.

A set $A$ together with a binary relation $E\subset A^2$ is called \emph{a directed graph}
(digraph). When examining digraphs, the elements of $A$ are called vertices and
a pair $(u,v)\in E$ is called the edge from $u$ to $v$. A \emph{loop} in a digraph is an
edge of the form $(u,u)$ for some $u\in A$.

A \emph{walk} from $u$ to $v$ in a digraph $E$ is a sequence of vertices
and edges of the form
\[
  u=w_1,e_1, w_2, e_2,\dots,e_{k-1},w_k=v
\]
such that for each $i$ the edge $e_i\in E$ is either $(w_i,w_{i+1})$ (a forward
edge) or $(w_{i+1},w_i)$ (a backward edge). 

The \emph{algebraic length} of a walk is the number of its forward edges
minus the number of its backward edges. (Algebraic length is 
not uniquely defined for walks with loops, but this will not be a
concern for our purposes.) A \emph{directed cycle} of length $k$
is a walk  from some $u$ to the same vertex $u$ that consists of $k$
forward edges and no backward edges.

A digraph has algebraic length 1 if there exists $u\in A$ and a walk
from $u$ back to $u$ of algebraic length 1. In particular, a digraph has algebraic
length 1 if it contains two directed cycles of lengths $n$ and $n-1$ that share
a vertex.

A digraph is \emph{smooth} if for each $a\in A$ there exists $b,c\in A$ such
that $(a,b),(c,a)\in E$.

We will need the following theorem. As an aside, we note that the
theorem has since its publication inspired many variants
and generalizations, often called ``loop lemmas.''
\begin{theorem}[\protect{\cite[Theorem 3.5]{barto-kozik-cyclic-terms-and-csp}}]\label{thmLoop}
  If a smooth digraph (on a finite set) has algebraic length 1 and admits a Taylor polymorphism then it contains a loop.
\end{theorem}
Here ``admits a Taylor polymorphism'' means that the relation $E$ is compatible with some
algebra $\algA$ that has a Taylor term operation.

In our proof we will need the following corollary of Theorem~\ref{thmLoop}. We include a
proof here because we have not found it in the published literature; we do not claim
originality (the result was known, e.g., to
M. Olšák~\cite{olsak-loop-conditions}).
\begin{corollary}\label{corqLoop}
  If a smooth digraph (on a finite set) has algebraic length 1 and
  admits a quasi Taylor polymorphism then it contains a loop.
\end{corollary}
\begin{proof}
  Denote the digraph in question by $(A,E)$. Let $\algA$ be an algebra with a
  quasi Taylor term such that $E$ is admissible for $\algA$.
  Let $\algB$ be the idempotent algebra for $\algA$ from 
  Lemma~\ref{lem:idempotent} and let $\alpha$ be a unary term
  operation of $\algA$ whose image is $B$ from the same Lemma.

  We know that $\algA$ satisfies some quasi Taylor identity $\Sigma$.
  Since this particular identity is a strong Maltsev condition, it
  follows that $\algB$ also satisfies $\Sigma$. Since $\algB$ is
  idempotent, it has a Taylor term.

  Consider the digraph with the edge relation $F=E\cap B^2$ on $B$. Since the
  operations of $\algB$ are just restrictions of (some) operations
  from the clone of $\algA$, it follows that $F$ is an admissible
  relation for $\algB$. Moreover, $F$ has algebraic
  length 1, which we can prove by taking the $\alpha$-image of any closed walk of
  algebraic length 1 in $(A,E)$.
  
  Since the digraph $(B,F)$ admits a Taylor polymorphism,
  Theorem~\ref{thmLoop} gives us a loop in $F$. Given that $F\subset
  E$, it
  follows that there is a loop in $E$ as well.
\end{proof}

\section{Deciding the existence of a $k$-qWNU for a fixed $k$}
\label{sec:main}
In this section we will define and provide an algorithm for two
computational problems.  In both of the problems the input algebra is
given by a list of tables of its (finitely many) basic operations.  The sum of the
sizes of these tables will be denoted by $\|\algA\|$.  We will assume
that the input algebra has at least one basic operation so that
$\|\algA\|\geq |A|$.

\begin{definition}
Define HAS-$k$-WNU-IDEMP to be the following decision problem:
\begin{description}
  \item[INPUT:] An \emph{idempotent} algebra~$\algA$ (on a finite set with
    finitely many basic operations).
  \item[QUESTION:] Does $\algA$ have a $k$-ary weak near unanimity operation?
\end{description}

Define HAS-$k$-qWNU to be the following decision problem:
\begin{description}
  \item [INPUT:] An algebra~$\algA$ (on a finite set with
    finitely many basic operations).
  \item [QUESTION:] Does $\algA$ have a $k$-ary quasi weak near unanimity operation?
\end{description}
\end{definition}
Note that in both problems the number $k$ is not a part of the input.
Indeed, the running time of our algorithm will depend exponentially on $k$ and 
we do not know if there is a polynomial time algorithm if $k$ is allowed to be
a part of the input (even if $k$ were written in the unary number system).
Also, both problems are trivial if $k=1$, so we will assume that
$k\geq 2$ in the rest of this section.

Note also that in HAS-$k$-WNU-IDEMP we demand that the input algebra
be idempotent. We do not know the complexity of HAS-$k$-WNU-IDEMP
should we drop the requirement that $\algA$ be idempotent, but we
guess that the problem is hard.

\begin{observation}\label{obsQuasi}
  The problem HAS-$k$-WNU-IDEMP reduces to HAS-$k$-qWNU.
\end{observation}
\begin{proof}
  Let $\algA$ be an idempotent algebra. Since a WNU is just an
  idempotent qWNU, it
  follows that $\algA$ has a $k$-WNU if and only if it has a $k$-qWNU.
  Therefore, we can just
  run HAS-$k$-qWNU with the input $\algA$ to solve HAS-$k$-WNU-IDEMP.
\end{proof}

Given Observation~\ref{obsQuasi}, it is enough to find a polynomial time
algorithm for HAS-$k$-qWNU.  Our strategy for that will be to show that having
local $k$-qWNU terms implies that the input algebra actually has a $k$-qWNU
term. This will give us a polynomial time algorithm since we can check the
presence of local $k$-qWNU terms is in polynomial time (see the proof of
Theorem~\ref{thmWNU} for details).  However, in order to use
Corollary~\ref{corqLoop} we will need to show that the algebra in question has
a Taylor term first.

\begin{definition}
  We say that an algebra $\algA$ has \emph{local quasi Taylor
  operations} if for
  every $a,b\in B$ there exists a term operation $t_{a,b}$ of $\algA$ such that we
  can replace the question marks below by either $a$ or $b$ so that the 
  following equalities hold
\begin{align*}
  t_{a,b}(a,?,?,\dots,?)&= t_{a,b}(b,?,?,\dots,?)\\
  t_{a,b}(?,a,?,\dots,?)&= t_{a,b}(?,b,?,\dots,?)\\
		  &\vdots\\
  t_{a,b}(?,?,?,\dots,a)&= t_{a,b}(?,?,?,\dots,b).
\end{align*}
  If in addition each $t_{a,b}$ can be chosen to be idempotent we say that
  $\algA$ has \emph{local Taylor operations}.
\end{definition}
It turns out that the local to global principle works for (quasi)
Taylor operations.
\begin{lemma}\label{lemTaylorIdempotent}
       Let $\algA$ be a finite idempotent algebra with local Taylor
       operations. Then
       $\algA$ has a Taylor operation.
     \end{lemma}
     \begin{proof}
      We use the fact that
    for a finite idempotent algebra $\algA$, having a Taylor term is
    equivalent to there \textbf{not} being a two element algebra in
    $HS(\algA)$ (the class of homomorphic images of subalgebras of
    $\algA$) whose term operations are projection maps.
       This was first established by A.
       Bulatov and P. Jeavons~\cite[Proposition 4.14]{bulatov-jeavons-hsa}; see also~\cite[Lemma
    9.4 and Theorem 9.6]{conm} and~\cite[Proposition
    3.1]{valeriote-intprop}.
       
       So, suppose to the contrary, that $HS(\algA)$ contains a
    two element algebra; let this algebra be the quotient of
       $\algB$ of $\algA$ by the congruence $\theta$. Take two elements $r$, $s \in B$ with $(r,s) \notin
    \theta$.  Let $t_{r,s}$ be a local Taylor term for for $r$ and $s$ in
       $\algA$.

    By assumption, the term $t_{r,s}$ on $\algB/\theta$ is a
    projection map, say onto its first variable (without loss of
       generality).  But then we have (the question marks stand for one of
       $r,s$ according to the first local Taylor equality for $t_{r,s}$)
  \begin{align*}
  s/\theta  &= t_{r,s}(s/\theta, ?/\theta, \dots, ?/\theta) =t_{r,s}(s, ?,
    \dots, ?)/\theta \\
  &=t_{r,s}(r,?,  \dots, ?)/\theta = t_{r,s}(r/\theta, ?/\theta, \dots, ?/\theta) = r/\theta,
  \end{align*}
  a contradiction.
     \end{proof}
\begin{corollary}\label{corQTaylor}
     Let $\algA$ be a finite algebra with local quasi Taylor
     operations. Then
$\algA$ has a quasi Taylor operation.
\end{corollary}
\begin{proof}
  Let $\algA$ be a finite algebra with local Taylor terms, but no quasi
  Taylor operation. Let $\algB$ be the algebra from
  Lemma~\ref{lem:idempotent} for $\algA$. Since $\algA$ and $\algB$ satisfy the
  same strong height 1 Maltsev conditions, if we show that
  $\algB$ has a Taylor operation, we shall get that $\algA$ has a quasi
  Taylor operation. 

  Given Lemma~\ref{lemTaylorIdempotent}, it is enough to show that $\algB$ has
  local Taylor terms. But that is easy: Let $b,c\in B$ and let $t_{b,c}$ be 
  the local quasi Taylor term for $b,c$ in $\algA$. By
  part~(\ref{itm:composition}) of Lemma~\ref{lem:idempotent} there
  exists a $\tau$ such that $\tau\circ t_{b,c}$ restricted to $B$ is
  a term operation of $\algB$. Applying $\tau$ to the both sides of
  the equalities
\begin{align*}
  t_{b,c}(a,?,?,\dots,?)&= t_{b,c}(b,?,?,\dots,?)\\
  t_{b,c}(?,a,?,\dots,?)&= t_{b,c}(?,b,?,\dots,?)\\
		  &\vdots\\
  t_{b,c}(?,?,?,\dots,a)&= t_{b,c}(?,?,?,\dots,b).
\end{align*}
  gives us the equalities
\begin{align*}
  \tau\circ t_{b,c}(a,?,?,\dots,?)&= \tau\circ t_{b,c}(b,?,?,\dots,?)\\
  \tau\circ t_{b,c}(?,a,?,\dots,?)&= \tau\circ t_{b,c}(?,b,?,\dots,?)\\
		  &\vdots\\
  \tau\circ t_{b,c}(?,?,?,\dots,a)&= \tau\circ t_{b,c}(?,?,?,\dots,b).
\end{align*}
  Since $\tau\circ t_{b,c}$ is idempotent by the choice of $\tau$, we
  see that $\tau\circ t_{b,c}$ is a local
  quasi Taylor operation for $b,c$ in $\algB$.
\end{proof}

\begin{definition}\label{defLocalKWNU}
Let $n\in \en$, $k\geq 2$.
  An algebra $\algA$ has \emph{$n$-local $k$-qWNU terms} if for every
      choice of $n$-tuples $\vect r,\vect s\in A^n$
      there is a $k$-ary term operation $t_{\vect r,\vect s}$ of $\algA$ such that
\[
t_{\vect r,\vect s}(\vect s, \vect r, \dots, \vect r)= t_{\vect r,\vect s}(\vect r, \vect s,\dots,\vect r)=\dots=t_{\vect r,\vect s}(\vect r, \vect r,\dots,\vect s),
\]
where the term operation  $t_{\vect r,\vect s}$ is applied coordinate-wise to
  the given $n$-tuples.  We call such a term an $n$-local $k$-qWNU term  of $\algA$ for $\vect s$ and $\vect r$.
  \end{definition}
Observe that $n$-local $k$-qWNU terms are a special case of local quasi Taylor terms.

        It is elementary to show that $\algA$ has $n$-local $k$-qWNU terms if and
	only if for each pair of $n$-tuples $\vect r,\vect s\in A^n$
      (which we will write as
      column tuples) the subuniverse $R$ of $\algA^{kn}$
      generated by the column vectors of the matrix
\[
   \left(\begin{matrix}
       \vect s&\vect r&\dots&\vect r&\vect r\\
       \vect r&\vect s&\dots&\vect r&\vect r\\
		     &&\ddots\\
       \vect r&\vect r&\dots&\vect s&\vect r\\
       \vect r&\vect r&\dots&\vect r&\vect s\\
     \end{matrix}
   \right)
     \]
     contains a tuple of the form $(\vect u,\vect u,\dots,\vect u)$ for some
     $\vect u\in A^n$.

  Furthermore, large enough local $k$-qWNUs are actually global: For
  an an algebra $\algA$ with the universe $A$ having $|A|^2$-local $k$-qWNU terms 
  is the same thing as having a $k$-qWNU term. All that remains is to bridge the gap between 1-local and $|A|^2$-local $k$-qWNU
     terms.
     
     \begin{lemma}\label{lemLocalKWNU}
    Let $\algA$ be a finite algebra and let $n \ge 1$. If $\algA$
    has $n$-local $k$-qWNU terms then
    $\algA$ also has $(n+1)$-local $k$-qWNU terms.
  \end{lemma}
  \begin{proof}
        Take $\vect r,\vect s\in A^{n}$ and $c,d\in A$.
    We want to show that
    the subpower of $\algA^{(n+1)k}$ generated by the columns of the
    matrix
 \[
   \begin{pmatrix}
       \vect s&\vect r&\dots&\vect r&\vect r\\
	c&d&\dots&d&d\\
       \vect r&\vect s&\dots&\vect r&\vect r\\
	d&c&\dots&d&d\\
		     &&\ddots\\
       \vect r&\vect r&\dots&\vect s&\vect r\\
	d&d&\dots&c&d\\
       \vect r&\vect r&\dots&\vect r&\vect s\\
	d&d&\dots&d&c\\
    \end{pmatrix}
       \]
       contains a tuple of the form $(\vect u, e,\vect u,e,\dots)$ for
       some $\vect u\in A^n$ and $e \in A$. By the $n$-local $k$-WNU property, we find that there
       is an $\vect a\in A^n$ and elements $b_1,\dots,b_k\in A$ such that
\[
   (\vect a,b_1, \vect a,b_2,\vect a,b_3,\dots,\vect a,b_k)\in R.
\]
    Inspired by the tuple above, let us consider the relation $S$ that we get from $R$ by the following
formula
\[
  S=\{(x_1,\dots,x_k)\colon \exists \vect c\in A^n,\, (\vect c,x_1,\vect
    c,x_2,\dots,\vect c,x_k)\in R\}.
\]
    It is easy to verify that $S$ is a subpower of $\algA$ (either
    from the definition or by observing that $S$ is defined from the
    admissible relation $R$ using a primitive positive formula).
We know that $(b_1,\dots,b_k)\in S$. Moreover, since we can permute
    the generators of $R$, it follows that if $(x_1,\dots,x_k)\in S$ and $\pi$ is a permutation of $[k]$, then
$(x_{\pi(1)},\dots,x_{\pi(k)})\in S$.  As an intermediate step in our proof, we will show that $S$ contains a tuple of the form $(e, e, \dots, e, f)$ for some $e$, $f \in A$.

We define the following digraph $G$: The vertex set of $G$ is
\[
  V=\{(v_1,\dots,v_{k-2})\colon \exists y,z\in A,\,
    (v_1,\dots,v_{k-2},y,z)\in S\}
\]
and the edge relation is
    \begin{align*}
      E=\{((v_1,v_2,\dots,v_{k-2}),&(v_2,v_3,\dots,v_{k-1}))\colon\\\exists z\in
      A,&\,
(v_1,v_2,\dots,v_{k-2},v_{k-1},z)\in S\}.
    \end{align*}
We want to show that $G$ contains a loop, since any loop of $G$ will witness that there
is a tuple in $S$ of the form $(e, e, \dots, e, f)$ for some $e$, $f \in A$. To find a loop, we
want to apply Corollary~\ref{corqLoop} to $G$. To do that we need to
    prove that $G$ is a smooth digraph of algebraic length 1 that
    admits a quasi Taylor polymorphism. 

    Let $t$ be a quasi Taylor term operation of $\algA$. As with $S$,
    it is easy to
    verify that both $V$ and $E$ are subpowers of $\algA$; hence $t$
    applied coordinate-wise is a quasi Taylor polymorphism of $G$.

    To see that $G$ is smooth, consider some $(v_1,\dots,v_{k-2})\in
    V$. By definition, there are $y,z$ such that
    $(v_1,\dots,v_{k-2},y,z)\in S$. By definition of $E$, we
    immediately get that there is an edge from
    $(v_1,\dots,v_{k-2})$ to $(v_2,\dots,v_{k-2},z)$. Since $S$
    is invariant under permutations, we also have
    $(z,v_1,\dots,v_{k-2},y)\in S$. Hence, there is an edge in $G$
    from $(z,v_1,\dots,v_{k-3})$ to $(v_1,\dots,v_{k-2})$, concluding
    the proof of smoothness of $G$.
    
To show that $G$ has algebraic length 1, we find two cycles of lengths
$k-1$ and $k$ that start at the vertex $(b_1, \dots, b_{k-2})$. Using $(b_1, \dots, b_k)\in S$,
we get the directed cycle of length $k$ in $G$ with vertices
\[
(b_1,\dots,b_{k-2}), (b_2,\dots,b_{k-1}),
(b_3,\dots,b_{k}),  \dots, (b_k,b_1,\dots,b_{k-3}).
\]
Since $S$ is invariant under permutations, we know that for each $m\in\en$ we also have
\[
(b_{1+m},b_{2+m},\dots,b_{k-1+m},b_k)\in S
\]
where the addition is modulo
$k-1$. This gives us the direct cycle of length $k-1$ in $G$ with vertices
\[
(b_1,\dots,b_{k-2}), (b_2,\dots,b_{k-1}),
(b_3,\dots,b_{k-1},b_1), \dots, (b_{k-1},b_1,\dots,b_{k-3}).
\]
Taken together, these two cycles imply that the algebraic length of $G$ is 1
and thus, by Theorem~\ref{thmLoop}, $G$ contains a loop and so $S$ contains a
tuple of the form $(e,\dots,e,f)$. The symmetry of $S$ gives us that $S$ contains all tuples of
the form $(e,e,\dots,e,f,e,\dots,e)$.  

Since $\algA$ has $n$-local $k$-qWNU operations, it follows that
$\algA$ has a $1$-local $k$-qWNU operations. Applying the 1-local $k$-qWNU
operation $t_{e,f}$ to $S$, we conclude that $S$ contains a constant tuple $(g,g, \dots, g)$ for
some $g \in A$.  From this, it follows that $R$ contains a tuple of
the form $(\vect c,g,\vect c,g,\dots,\vect c,g)$, showing that $\algA$
has $(n+1)$-local $k$-WNU terms and we are done.
\end{proof}

\begin{theorem}\label{thmWNU}
  The problem HAS-$k$-qWNU is solvable in polynomial
  time.
\end{theorem}
\begin{proof}
  By Lemma~\ref{lemLocalKWNU} we just need to test if $\algA$ has
  1-local $k$-qWNU terms. As noted just after
  Definition~\ref{defLocalKWNU}, this amounts to testing if for each
  $r$, $s \in A$ the subpower of $\algA$ generated by the $k$-tuples
  \[
  (s,r,r, \dots, r), (r,s, r, \dots, r), \dots, (r, r, r, \dots, s)\]
  contains a constant tuple.  By \cite[Proposition 6.1]{freese-valeriote-maltsev-conditions}, we can generate this
  subpower (for fixed $r,s$) by an algorithm whose run-time is
  $O(m\|\algA\|^k)$. Here $m$ is the largest arity of a basic operation of
  $\algA$. Since this test needs to be performed for each pair of elements
  $(r,s)$ from $A^2$, we conclude that testing for a $k$-qWNU term can be
  carried out by an algorithm whose run-time is $O(m|A|^2\|\algA\|^k)$, which
  is polynomial in $\|\algA\|$.
\end{proof}

\begin{remark}
  One might wonder whether the structural results in
  Lemma~\ref{lemTaylorIdempotent}, Corollary~\ref{corQTaylor} and
  Lemma~\ref{lemLocalKWNU} remain true when we drop the requirement that $\algA$ is
  finite. We give a counterexample to the infinite versions of 
  Lemma~\ref{lemTaylorIdempotent} and Corollary~\ref{corQTaylor}. We suspect that
  the infinite version of Lemma~\ref{lemLocalKWNU} also fails to be true, but we
  do not have a proof.
\end{remark}

\begin{proposition}
  There exists an infinite idempotent algebra with binary local Taylor
  operations but without a global Taylor operation.
\end{proposition}
\begin{proof}
  Consider the algebra
  $\algA$ on
  the set of nonnegative integers with one binary operation $f$ defined as:
    \[
    f(x,y)=\begin{cases}
      x&\text{if $x=y$,}\\
      \max(1,x-1)&\text{else.}
    \end{cases}
    \]
  To see that $\algA$ has local Taylor terms, consider any pair of distinct
  nonnegative integers $a,b$ and let $E$ be the
  the subuniverse of $\algA^2$ generated by $(a,b)$ and $(b,a)$. Applying $f$
  to the generators, we see that $E$ contains the tuple $(\max(a-1,1),\max(b-1,1))$. Applying $f$ to
  $(\max(a-1,1),\max(b-1,1))$ and $(a,b)$, we obtain that $E$ contains $(\max(a-2,1),\max(b-2,1))$. Continuing in
  this manner, we eventually get that $E$ contains the pair $(1,1)$, yielding 
  a binary local Taylor operation for $a$ and $b$. 

  To show that $\algA$ has no global Taylor term, we need to examine the the
  term operations of $\algA$ first. It is easy to see that $f(x,y)\in
  \{x,x-1\}$ for any $x,y\in\en$. From this it follows by induction on term
  complexity that when $t$ is an operation of $\algA$ whose term contains $n$
  occurrences of the symbol $f$ then for all $x_1,\dots,x_n$ we have
  $t(x_1,\dots,x_n)\in\{x_i-n,x_i-n+1,\dots,x_i\}$ where $x_i$ is the leftmost
  variable in a term for $t$.

  Let $t$ be a term operation of $\algA$ whose term contains $n$ occurrences of
  $f$. Assume without loss of generality that $x_1$ is the leftmost
  variable in the term for $t$. We now claim that $t$ fails to satisfy any Taylor term
  identity of the form $t(x,?,\dots,?)\approx t(y,?,\dots,?)$. To see this, let
  us choose $x,y\in \en$ so that $x-n>y$. By the previous paragraph we have
  for any choice of $x$ or $y$ in place of question marks
  \begin{align*}
    t(x,?,\dots,?)&\in \{x-n,x-n+1,\dots,x\}\\ 
    t(y,?,\dots,?)&\in
    \{y-n,y-n+1,\dots,y\}.
  \end{align*}
  However, since $x-n>y$, we have that the sets 
$\{x-n,x-n+1,\dots,x\}$ and  $\{y-n,y-n+1,\dots,y\}$ are disjoint and so the
identity $t(x,?,\dots,?)\approx t(y,?,\dots,?)$ cannot hold no matter what the
  question marks are. We conclude that $\algA$ has no Taylor terms.
\end{proof}

\section{Deciding quasi Taylor terms}
We end our paper by an observation that connects our result to a
related problem of deciding if an input algebra has a (quasi) WNU operation.

\begin{definition}
Define HAS-TAYLOR-IDEMP to be the following decision problem:
\begin{description}
  \item[INPUT:] An \emph{idempotent} algebra~$\algA$ (on a finite set with
    finitely many basic operations).
  \item[QUESTION:] Does $\algA$ have a Taylor term?
\end{description}

Define HAS-QTAYLOR to be the following decision problem:
\begin{description}
  \item[INPUT:] An algebra~$\algA$ (on a finite set with
    finitely many basic operations).
  \item[QUESTION:] Does $\algA$ have a quasi Taylor term?
   \end{description}
\end{definition}

Note that by Theorem~\ref{thmTaylorOmnibus}, the ``yes'' instances of
HAS-TAYLOR-IDEMP are exactly
those idempotent algebras that have some WNU operation and the ``yes'' instances
of HAS-QTAYLOR are exactly algebras with some qWNU operations.
However, HAS-QTAYLOR is a different problem than the problem of
deciding whether the variety generated by $\algA$ omits type 1.
Omitting type 1 is equivalent to having a Taylor term and this problem is 
$\compEXPTIME$-complete~\cite[Corollary
9.3]{freese-valeriote-maltsev-conditions}.

The problem HAS-TAYLOR-IDEMP is known to be in $\compP$~\cite[Theorem
6.3]{freese-valeriote-maltsev-conditions}.  However, should we drop
the requirement that the input algebra be idempotent, HAS-TAYLOR-IDEMP
would become $\compEXPTIME$-complete as mentioned in the previous
paragraph. We will show by combining
several known results that HAS-QTAYLOR lies the class $\compP$.

\begin{observation}\label{obsSiggers}
  A finite algebra $\algA$ has a quasi Taylor operation if and only if for
  every $a,b\in A$ the subalgebra of $\algA^4$ generated by the columns of the
  matrix
\[
   \begin{pmatrix}
     a&b&b&b\\
     b&a&b&b\\
     b&b&a&b\\
     b&b&b&a\\
      \end{pmatrix}
 \]
  contains a tuple of the form $(q,q,r,r)$ for some $q,r\in A$.
\end{observation}
\begin{proof} Assume first that $\algA$ has a quasi Taylor operation.
  By Theorem~\ref{thmTaylorOmnibus} 
  the algebra $\algA$ also has a
  quasi Siggers operation $s$ that satisfies the identity $s(r,a,r,e)\approx s(a,r,e,a)$. It is easy to
  verify that from this identity it follows 
  \begin{align*}
    s(a,b,b,b)&\approx s(b,a,b,b)\\
    s(b,b,a,b)&\approx s(b,b,b,a).
  \end{align*}
  Therefore, given $a,b\in A$ we can apply $s$ to the columns of
  the matrix
\[
   \begin{pmatrix}
     a&b&b&b\\
     b&a&b&b\\
     b&b&a&b\\
     b&b&b&a\\
   \end{pmatrix}
 \]
  to get a tuple whose first two and last two entries are identical.

  In the other direction, the terms $s_{a,b}$ that witness the existence of a tuple
  $(q,q,r,r)$ in the subpower of $\algA$ generated by the columns of
 \[
   \begin{pmatrix}
     a&b&b&b\\
     b&a&b&b\\
     b&b&a&b\\
     b&b&b&a\\
      \end{pmatrix}
 \]
  are local quasi Taylor terms. Therefore $\algA$
  has a quasi Taylor operation by Corollary~\ref{corQTaylor}.
\end{proof}
\begin{corollary}
  The following algorithm solves the problem HAS-QTAYLOR in time
  polynomial in $\|\algA\|$: For each $a,b\in A$ examine the subpower
  $R_{a,b}$
  generated by the columns of the matrix
\[
   \begin{pmatrix}
     a&b&b&b\\
     b&a&b&b\\
     b&b&a&b\\
     b&b&b&a\\
      \end{pmatrix}.
 \]
  If for some $a,b$ the relation $R_{a,b}$ does not contain a tuple of
  the form $(q,q,r,r)$, answer ``no.'' Else answer ``yes.''
\end{corollary}
\begin{proof}
  The correctness of the algorithm follows from
  Observation~\ref{obsSiggers} while the analysis of the running time
  is similar to that done in the proof of Theorem~\ref{thmWNU}.
\end{proof}

\section{Acknowledgments}
The author wishes to thank the anonymous referee for their interesting comments.

This work has been supported by the Czech Science Foundation project
GA ČR 18-20123S as well as  the INTER-EXCELLENCE project
LTAUSA19070 M\v SMT Czech Republic and by the Charles University Research Centre
program number UNCE/SCI/022.

\bibliography{citations}

\begin{thebibliography}{10}

\bibitem{barto-kozik-cyclic-terms-and-csp}
Libor Barto and Marcin Kozik.
\newblock (2012).
\newblock Absorbing subalgebras, cyclic terms, and the constraint satisfaction
  problem.
\newblock {\em Logical Methods in Computer Science}, 8(1).

\bibitem{wonderland}
Libor Barto, Jakub~Opr\v sal, and Michael Pinsker.
\newblock (2018).
\newblock The wonderland of reflections.
\newblock {\em Israel Journal of Mathematics}, 1(223):363--398.

\bibitem{bulatov-jeavons-hsa}
Andrei~A. Bulatov and Peter Jeavons.
\newblock (2001).
\newblock Algebraic structures in combinatorial problems.
\newblock Technical Report MATH-AL-4-2001, Technische Universit{\" a}t Dresden.

\bibitem{burris-sankapanavar}
Stanley Burris and H.~P. Sankappanavar.
\newblock (2012).
\newblock {\em A Course in Universal Algebra}.
\newblock Springer, New York, graduate texts in mathematics edition.
\newblock The {M}illenium {E}dition (electronic book).

\bibitem{chen-larose-meta}
Hubie Chen and Benoit Larose.
\newblock (2017).
\newblock Asking the metaquestions in constraint tractability.
\newblock {\em ACM Transactions on Computation Theory}, 9(3).

\bibitem{UACalc}
Ralph Freese, Emil Kiss, and Matthew Valeriote, (2011).
\newblock Universal {A}lgebra {C}alculator.
\newblock Available at: {\verb+www.uacalc.org+}.

\bibitem{freese-valeriote-maltsev-conditions}
Ralph Freese and Matthew Valeriote.
\newblock (2009).
\newblock On the complexity of some {M}altsev conditions.
\newblock {\em International Journal of Algebra and Computation},
  19(01):41--77.

\bibitem{conm}
David Hobby and Ralph McKenzie.
\newblock (1996).
\newblock {\em The Structure of Finite Algebras}, volume~76 of {\em
  Contemporary Mathematics}.
\newblock American Mathematical Society, Providence.

\bibitem{horowitz-ijac}
Jonah Horowitz.
\newblock (2013).
\newblock Computational complexity of various {M}al'cev conditions.
\newblock {\em International Journal of Algebra and Computation}, 23(6):1521.

\bibitem{minority-report}
Alexandr Kazda, Jakub Opršal, Matthew Valeriote, and Dmitriy Zhuk.
\newblock (2020).
\newblock Deciding the existence of minority terms.
\newblock {\em Canadian Mathematical Bulletin}, 63(3):577--591.

\bibitem{kazda-valeriote}
Alexandr Kazda and Matthew Valeriote.
\newblock (2020).
\newblock Deciding some {M}altsev conditions in finite idempotent algebras.
\newblock {\em The Journal of Symbolic Logic}.
\newblock Accepted.

\bibitem{maroti-mckenzie-wnu}
Mikl\'os Mar\'oti and Ralph McKenzie.
\newblock (2008).
\newblock Existence theorems for weakly symmetric operations.
\newblock {\em Algebra Universalis}, 59:463--489.

\bibitem{olsak-loop-conditions}
Miroslav Ol\v{s}\'{a}k.
\newblock (2020).
\newblock Loop conditions for strongly connected digraphs.
\newblock {\em International Journal of Algebra and Computation},
  30(03):467--499.

\bibitem{siggers-original}
Mark~H. Siggers.
\newblock (2010).
\newblock A strong {M}al'cev condition for locally finite varieties omitting
  the unary type.
\newblock {\em Algebra universalis}, 64(1):15--20.

\bibitem{taylor}
Walter Taylor.
\newblock (1977).
\newblock Varieties obeying homotopy laws.
\newblock {\em Canadian Journal of Mathematics}, 29(3):498--527.

\bibitem{valeriote-intprop}
Matthew Valeriote.
\newblock (2009).
\newblock A subalgebra intersection property for congruence distributive
  varieties.
\newblock {\em Canadian Journal of Mathematics}, 61(2):451–464.

\end{thebibliography}

\end{document}